\documentclass[11pt,twoside]{amsart}

\usepackage{amsmath, amsthm, amssymb}
\usepackage{array}
\usepackage{ytableau}
\usepackage{hyperref}

\usepackage[normalem]{ulem}
\usepackage{color}

\input xy
\xyoption{all}

\DeclareMathOperator{\SB}{SB} 
\DeclareMathOperator{\CH}{CH}  
\DeclareMathOperator{\K0}{K_0}   
\DeclareMathOperator{\codim}{codim} 
\DeclareMathOperator{\im}{im} 
\DeclareMathOperator{\ind}{ind} 
\DeclareMathOperator{\id}{id} 
\DeclareMathOperator{\res}{res} 
\DeclareMathOperator{\rrank}{rrank}  
\DeclareMathOperator{\rk}{rk} 
\DeclareMathOperator{\coker}{coker}  
\DeclareMathOperator{\Br}{Br} 
\DeclareMathOperator{\End}{End} 
\DeclareMathOperator{\Hom}{Hom} 
\DeclareMathOperator{\Gr}{Gr}  
\DeclareMathOperator{\Pic}{Pic}  
\DeclareMathOperator{\PGL}{PGL} 
\DeclareMathOperator{\Sing}{Sing} 
\DeclareMathOperator{\Gal}{Gal} 
\DeclareMathOperator{\Tors}{Tors} 
\DeclareMathOperator{\diag}{diag} 

\newcommand{\ZZ}{\mathbb{Z}}   
\newcommand{\Z}{\mathbb{Z}}  
\newcommand{\Pj}{\mathbb{P}}  
\newcommand{\Gm}{\mathbb{G}_m}  
\newcommand{\EE}{\mathcal{E}}  
\newcommand{\LL}{\mathcal{L}}   
\newcommand{\fA}{\mathfrak{A}} 
\newcommand{\fB}{\mathfrak{B}} 
\newcommand{\rA}{\mathbf{A}} 
\newcommand{\into}{\hookrightarrow} 

\newtheorem{thm}{Theorem}[section]
\newtheorem*{thm*}{Theorem}
\newtheorem{prop}[thm]{Proposition}
\newtheorem*{prop*}{Proposition}
\newtheorem{lem}[thm]{Lemma}
\newtheorem{cor}[thm]{Corollary}
\newtheorem*{cor*}{Corollary}

\theoremstyle{definition}   
\newtheorem{defn}[thm]{Definition}
\newtheorem{rem}[thm]{Remark}

\newtheorem*{ex*}{Example}

\title{Schubert cycles and subvarieties of generalized Severi-Brauer varieties}
\author{Caroline Junkins}
\author{Daniel Krashen}
\author{Nicole Lemire}
\date{}

\begin{document}

\maketitle

\section{Introduction}

For a central simple algebra $A$ of degree $n$ over a field $F$, the generalized Severi-Brauer variety $\SB(d, A)$ is a
twisted form of $\Gr(d, n)$, the Grassmannian of $d$-dimensional planes in
$n$-dimensional affine space. It is well-known that the variety $\SB(d, A)$
has a rational point over an extension $K/F$ if and only if $\ind(A_K)\mid
d$. We can extend this question to ask about other closed subvarieties of
$\SB(d, A)$. In particular, we may ask: Under what conditions does $\SB(d,
A)$ contain a closed subvariety which is a twisted form of a Schubert
subvariety of $\Gr(d, n)$? We show in this paper that this happens
exactly when the index of the algebra divides a certain number arising
from the combinatorics of the Schubert cell, using a variation on Fulton's
notion of essential set of a partition (see Definition~\ref{corner def},
\ref{corner dim def}):
\begin{thm*}[\ref{thm:all}]
Let $A$ be a central simple algebra, and $K/F$ a splitting field for $A$.
Then the generalized Severi-Brauer variety $\SB(d, A)$ has a closed subvariety
$P$ such that $P\otimes_F K\simeq X_{\lambda}$ for a Schubert subvariety
$X_\lambda$ if and only if $\ind(A)\mid \gcd(\overline{E_\lambda})$.
Moreover, in this case, $A$ contains a flag of right ideals $I_{a_1}\subset
\cdots \subset I_{a_r}$ for $\overline{E_{\lambda}}=\{a_1, \dots, a_r\}$ such that
for any finite extension $L/F$, 
$$
P(L)=\{J\subseteq A_L: \rk(J\cap (I_a)_L)\ge j \mbox{ for } (j,a)\in
E_{\lambda}\}
$$
\end{thm*}

In the classical setting, these Schubert subvarieties are of particular interest as they form the building
blocks for the Grothendieck group and Chow group of $\Gr(d, n)$. The
Chow groups, in the case of homogeneous varieties, and Severi-Brauer varieties
in particular, have been much studied, and
related to important questions about the arithmetic of central simple
algebras. Although the Chow
groups of dimension $0$, codimension $1$ and to some extent codimension $2$
cycles on Severi-Brauer varieties have been amenable to study, the other
groups are in general not very well understood at all. 

Algebraic cycles on and Chow groups of
generalized Severi-Brauer varieties are even more subtle and less understood. For example, the Chow
group of dimension $0$ cycles on such varieties are only known in the
case of reduced dimension $2$ ideals in algebras of period $2$
\cite{CM06, Kr10}.

In this paper we make some first steps towards developing parallel methods
as currently exist for the Severi-Brauer varieties, to compute the
codimension $2$ Chow groups for the generalized Severi-Brauer varieties of
reduced dimension $2$ ideals in certain algebras of small index.

\begin{thm*} \ref{cycle thm}
Let $X=\SB(2,A)$ with $\ind(A)|12$. Then $\CH^2(X)$ is torsion-free.
\end{thm*}

This computation is done first by using the explicit descriptions of
Schubert classes obtained in the first half part of the paper together with
other geometric constructions to show that the graded pieces of the
$K$-groups with respect to the topological filtration are torsion free
\ref{prop:plucker} for degree $4$ algebras. This quickly gives the result for
codimension $2$ Chow groups for such varieties. Finally,
the theorem follows by an analysis of the motivic decomposition of the Chow
motive of $\SB(2,A)$ due to Brosnan.


\section{Schubert varieties and their twisted forms}


\subsection{Generalized Severi-Brauer varieties}

Let $A$ be a central simple algebra over a field
$F$. By Wedderburn's theorem, such an
algebra can be written uniquely as $A\simeq \End_D(V)$ for a division
algebra $D$ over $F$ and a vector space $V$ over $D$. The degree of $A$ is
defined by $\deg(A)=\sqrt{\dim_F(A)}$ and the index of $A$ is defined by
$\ind(A)=\deg(D)$.

If $A=\End_D(V)$, there is a one-to-one correspondence between the right
ideals of $A$ and the vector subspaces of $V$ over $D$, given by sending a
subspace $W\subseteq V$ to the right ideal $\Hom_D(V, W)$. Defining the
reduced rank of a right ideal $I$ of $A$ by
$\rrank(I):=\frac{\dim_F(I)}{\deg(A)}$, for $W\subseteq V$ with
$\dim_D(W)=k$, we obtain $\rrank(\Hom_D(V, W))=k\cdot \ind(A)$.

The generalized Severi-Brauer variety $\SB(d, A)$ is defined to be the
variety with $L$-points given by
$$
\SB(d, A)(L)=\{I\subset A_L \mid \rrank(I)=d\}
$$ 

Note that in the case $A = \End(V)$, for an $F$-vector space $V$, the
above gives us a natural bijection between the right ideals of $A$ and the
linear subspaces $W \subset V$. It follows that for
an arbitrary $A$, the variety $\SB(d, A)$ is a twisted form of $\Gr(d,
n)$, the Grassmannian of $d$-dimensional planes in $n$-dimensional affine
space. By \cite{Bl}, $\SB(d, A)$ has a rational point over an extension
$K/F$ if and only if $\ind(A_K)\mid d$.


\subsection{Twisted forms of Schubert varieties}

Consider a partition $\lambda=[\lambda_1, \dots, \lambda_d]$ with $n-d\geq
\lambda_1\geq \lambda_2\geq \cdots \geq \lambda_d\geq  0$. Such a partition
can be represented by a Young diagram with $\lambda_j$ boxes in the $j$-th
row.

Given a split central simple algebra $A=\End_F(V)$ of degree $n$, we
fix a chain of ideals 
$$
I_1\subset I_2\subset \dots \subset I_n=A\;\;\;\;\;\text{such that $\rrank(I_j)=j$.}
$$ 

The Schubert variety $X_{\lambda}\subseteq \SB(d, \End_F(V))$ is defined by intersection conditions with respect to this chain:
$$
X_{\lambda}=\{ J\in\SB(d, A) \mid \rrank(J\cap I_{n-d+j-\lambda_j})\geq j \text{ for $j=1,\dots, d$}\}
$$

\begin{defn}
For a central simple $F$-algebra $A$, and vector space $V$ as above,
we say that a subvariety $P \subset SB(d, A)$ is a twisted form of 
the Schubert variety $X_\lambda$ if we can find an isomorphism $A_{\overline
F}$ such that $P_{\overline F}$ is identified with $X_\lambda$.
\end{defn}
Our first main goal, as described in the introduction, will be to determine
which twisted forms of Schubert varieties appear in the generalized
Severi-Brauer varieties of a given algebra $A$.

We introduce a variation on Fulton's \emph{essential
set}, defined in \cite{Fu92}: 
\begin{defn} \label{corner def}
For each partition $\lambda$, we define a set
of pairs $E_\lambda:=\{(j, n-d+j-\lambda_j) \mid
\text{$\lambda_j>\lambda_{j+1}$}\}$ with the convention that
$\lambda_{d+1}=0$ for any partition $\lambda$. 
\end{defn}
We note that the condition
$\lambda_j > \lambda_{j+1}$ describes the condition that the right-most box
in the $j$'th row of the corresponding Young diagram lies at the
south-eastern edge of a ``corner.'' The second coordinate of these pairs
records the reduced rank of the right ideal of the flag which defines the
intersection condition.
\begin{defn} \label{corner dim def}
We define
$\overline{E_\lambda}$ to be the set of distinct integers given by
projection onto the second coordinate, i.e. $\overline{E_\lambda}=\{a \mid
(j, a)\in E_\lambda\}$.
\end{defn}
Note that for any $a\in\overline{E_\lambda}$, there
exists a unique $1\leq j\leq d$ such that $(j, a)\in E_\lambda$.

The set $E_\lambda$ defines a subchain $I_{a_1}\subset I_{a_2}\subset\dots\subset I_{a_m}$, $a_i\in \overline{E_\lambda}$ with the property that 
$$
X_\lambda=\{ J\in \SB(d, A) \mid \rrank(J \cap I_{a})\geq j \text{ for $(j, a)\in E_\lambda$}\}
$$

Thus, $E_\lambda$ uniquely defines $X_\lambda$ and is minimal
in the sense
that if any pairs in $E_\lambda$ are changed or removed, the variety
defined by this new set of conditions will not be equal to $X_\lambda$
(see \cite[Lemma~3.14]{Fu92}).  Note that the Young diagram for $\lambda$ can
also be reconstructed from $E_\lambda$.

\begin{prop}\label{prop:easy} 
Let $X_\lambda$ be a Schubert subvariety of $\Gr(d, n)$ and let $A$ be a central simple $F$-algebra of degree $n$. If $\ind(A)\mid\gcd(\overline{E_\lambda})$, then $\SB(d, A)$ contains a closed subvariety over $F$ which is a twisted form of $X_\lambda$. 
\end{prop}

\begin{proof}
Since $A$ has a right ideal of reduced rank $k$ if and only if $\ind(A)\mid
k$, it follows that if $\ind(A)\mid \gcd(\overline{E_\lambda})$, then $A$
contains a partial flag of ideals $I_{a_1}\subset \cdots\subset I_{a_r}\subset A$ for $\overline{E_\lambda}=\{a_1, \dots, a_r\}$. 

We may then define a closed subvariety $P_\lambda\subseteq \SB(d, A)$ by 
\begin{equation}\label{intersections}
P_\lambda(L)=\{J\in \SB(d, A_L) \mid \rrank(J\cap (I_{a}\otimes_F L))\geq j \text{ for $(j, a)\in E_\lambda$}\} 
\end{equation}

\end{proof}















The remainder of this section is devoted to proving the converse of Proposition \ref{prop:easy}. That is, we want to show that if $\SB(d, A)$ has a closed subvariety $P$ defined over $F$ which is a twisted form of a Schubert variety $X_\lambda$, then we must have $\ind(A)\mid \gcd(\overline{E_\lambda})$. More specifically, we would like to show that $P$ must be defined via intersection conditions with right ideals of $A$ as in \eqref{intersections}. In the case of usual Severi-Brauer varieties (i.e. $d=1$), this question was first answered by Artin.

\begin{thm}[\cite{Ar82}]\label{thm:Artin}
$\SB(A)$ has a closed subvariety $P$ such that $P\otimes_F K\simeq \Pj^{m-1}$  if and only if $\ind(A)\mid m$. Moreover, $P=\SB(B)$ for a central simple $F$-algebra $B$ such that $[A]=[B]\in\Br(F)$. 
\end{thm}

Recall that such a Severi-Brauer variety $\SB(A)$ has the unique property (among generalized Severi-Brauer varieties) that each Schubert subvariety of $\SB(A)$ is also a twisted form of projective space, i.e. is also a Severi-Brauer variety. Considering the closest analogue to this situation for a generalized Severi-Brauer variety $\SB(d, A)$, we ask first about the Schubert subvarieties of $\SB(d, A)$ which are twisted forms of subgrassmannians $\Gr(d, m)$, $m\leq n$.

Using a generalization of Artin's argument, this question was answered by
the second author in \cite{Kr08}.

\begin{thm}[\cite{Kr08}, Thm. 2.2] \label{thm:subgrassmannian}
$\SB(d, A)$ has a closed subvariety $P$ such that $P\otimes_F K\simeq \Gr(d, m)$ for some $m\leq n$ if and only if $\ind(A)\mid m$. Moreover, $P=\SB(d, B)$ for a central simple $F$-algebra $B$ such that $[A]=[B]\in\Br(F)$. 
\end{thm}

To place these results in the context of our main question, note that for
any $d\leq m\leq n$, the subgrassmannian $\Gr(d, m)$ is equal to the
Schubert subvariety $X_\lambda$ for $\lambda=[n-m,  \dots, n-m] = [(n-m)^d]$ with $d$
copies of $n-m$ and $E_\lambda=\{(d, m)\}$. If $\SB(d, B)$ is a subvariety
of $\SB(d, A)$ defined over $F$ with $\SB(d, B)\otimes_F K\simeq \Gr(d,
m)$, then $B\simeq\End_A(I_m)$ for a right ideal $I_m\subset A$ of reduced rank $k$ and we may write 
\begin{align}
P(L)&=\{J\in \SB(d, A_L) \mid \rrank(J\cap (I_m\otimes_F L))\geq d\}\\
&=\{J\in \SB(d, A_L) \mid J\subseteq I_m\otimes_F L\} \nonumber
\end{align}

Another Schubert variety which can be described via a single right ideal of
$A$ is found by replacing ``contained in" above by ``containing" for the
intersection conditions. Note that the Schubert subvariety $X_{\lambda}$ of
the Grassmannian $\Gr(d,n)$ for $\lambda=[(n-d)^m,0^{d-m}]$ for some $1\le
m\le d$ consists of all $d$ dimensional subspaces of $n$ dimensional space
which contain a fixed $m$ dimensional subspace. This leads us to one
further generalization of Artin's argument.

\begin{prop}\label{prop:containing}
Consider a generalized Severi-Brauer variety $\SB(d, A)$ with a
Galois splitting
field $K/F$. For any $1\le m\le d$, there exists a closed subvariety
$P\subset \SB(d, A)$ defined over $F$ such that $P\otimes_F K=X_{\lambda}$
for $\lambda=[(n-d)^m, 0^{d-m}]$ if and only if $\ind(A)\mid m$. Moreover,
$P=\SB(d-m, B)$ for some central simple $F$-algebra $B$ such that
$[A]=-[B]\in\Br(F)$ (where $m$ is the number of nonzero entries in
$\lambda$).
\end{prop}

\begin{proof}
Suppose $P$ is a subvariety of $\SB(d, A)$ such that $P_K \cong
X_\lambda$, for
a partition $\lambda=[(n-d)^m, 0^{d -m}]$. Note that
$E_\lambda=\{(m, m)\}$, and $X_{\lambda}$ is described as
the set of reduced rank $d$ right ideals of $A_K\simeq\End_K(V)$
containing the right ideal $\Hom_K(V, W)$ for some vector subspace
$W\subseteq V$ with $\dim_K(W)=m$.

Consider the right ideal of $End_K(V)$ defined as
$$
\overline I:=\bigcap_{J\in P_K(K)}J
$$

Notice that $P$ is fixed by $\Gal(K/F)$, since the collection of $J$ in the
indexing set above are permuted by the Galois action. It follows by descent
that $\overline{I} = I_K$ for some right ideal $I$ of $A$. Again by
descent, it follows that $I$ must have reduced rank $m$. 

Thus, for any extension $L/F$, we have the following description of the $L$-points of $P$:
\begin{align}
P(L)&=\{ J\in\SB(d, A_L) \mid \rrank(J\cap (I\otimes_F L))\geq m\}\\
&=\{J\in\SB(d, A_L)\mid J\supseteq I\otimes_F L\} \nonumber
\end{align}
Let $B = End_A(I^0)$ be the algebra of left $A$-linear endomorphisms of the
left annihilator of the right ideal $I$. Note that we can also write $B =
C_{End_F(I^0)}(A)$, from which it follows that $B$ is Brauer equivalent to
$A^{op}$, the opposite algebra. 

By \cite[Proposition~1.20]{KMRT}, for any right ideal
$$
\SB(d-m, \End_A(I^0))\into \SB(d, A)
$$ 
whose image is the variety of right ideals of reduced rank $d$ in $A$
which contain $I$. 

It follows that there is an induced map $\SB(d-m, \End_A(I^0)) \to P$ which
is an isomorphism on $L$-points for every field extension $L/F$. Since this
is a map of varieties, it follows that it is an isomorphism.
\end{proof}

By combining the results of Theorem \ref{thm:subgrassmannian} and Proposition \ref{prop:containing}, we can extend our classification to a larger set of Schubert subvarieties. We say that a subvariety $P\subset \SB(d, \End_F(V))$ is ``defined by inclusions'' if it can be defined as 
$$
P=\{J\in \SB(d, \End_F(V)) \mid I\subseteq J \subseteq I'\}
$$
for right ideals $I, I'$ of $\End_F(V)$ such that $0\leq \rrank(I)\leq d$
and $d\leq \rrank(I')\leq n$. Such varieties properly characterize the
smooth Schubert varieties, namely, if $\rrank(I)=k$ and $\rrank(I')=m$, then $P=X_\lambda$ for some flag, with $\lambda=[n-d, \dots, n-d, n-m, \dots, n-m]$, with $k$ copies of $n-d$ (c.f. \cite{LW90}). In terms of the essential set of pairs, $X_\lambda$ is a smooth Schubert subvariety of $\Gr(d, n)$ if and only if $E_\lambda\subseteq \{(k, k), (d, m)\}$ for some $k\leq d\leq m$.

\begin{thm} \label{thm:smooth}
Let $X_\lambda$ be a smooth Schubert subvariety of $\Gr(d, n)$. Then, $\SB(d, A)$ has a closed subvariety $P$ such that $P\otimes_F K\simeq X_{\lambda}$ if and only if $\ind(A)\mid \gcd(\overline{E_\lambda})$. Moreover, there exist right ideals $I_k, I_m\subseteq A$ such that
$$
P(L)=\{J\in \SB(d, A_L) \mid (I_k\otimes_F L)\subseteq J \subseteq (I_m\otimes_F L)\}
$$
\end{thm}

\begin{proof}
Without loss of generality, we may assume $E_\lambda=\{(k, k), (d, m)\}$ for some $k<d<m$. As in the proofs of Theorem \ref{thm:subgrassmannian} and Proposition \ref{prop:containing}, we construct the following right ideals of $A_K$:
$$
I:=\bigcap_{J\in P_K(K)}J \;\;\;\;\;\;\;\;\; \text{ and } \;\;\;\;\;\;\;\;\; I':=\sum_{J\in P_K(K)}J
$$
Both of these ideals are fixed by $\Gal(K/F)$ and so we must have
$I=I_k\otimes_F K$ and $I'=I_m\otimes_F K$ for right ideals $I_k,
I_m\subset A$ of reduced ranks $k$ and $m$ respectively.

For any extension $L/F$, we have the following description of the $L$-points of $P$:
\begin{align}
P(L)&=\{ J\in\SB(d, A_L) \mid \rrank(J\cap (I_k\otimes_F L))\geq k \text{
and } \rrank(J\cap(I_m\otimes_F L))\geq d\}\\
&=\{J\in\SB(d, A_L)\mid (I_k\otimes_F L)\subset J\subset (I_m\otimes_F L)\} \nonumber
\end{align}
\end{proof}

At this point, we have extended Artin's arguments to provide a converse for
Proposition \ref{prop:easy} for all smooth Schubert forms of $\SB(d, A)$.
These arguments do not directly extend to the remaining Schubert forms since these varieties are not defined by inclusions and as such we cannot construct right ideals of $A$ by taking spans or intersections. In the following section we use the composition of the singular locus of a Schubert variety to construct these ideals after some combinatorial manipulations.


\section{The singular locus of a Schubert variety}

\ytableausetup{boxsize=0.35cm}


\subsection{Definition and Preliminaries}

Fixing $n$ and $d$, consider a partition $\lambda=[\lambda_1, \dots, \lambda_d]$ which corresponds to a singular Schubert subvariety $X_\lambda\subset \Gr(d, n)$. The singular locus $\Sing(X_\lambda)$ of $X_\lambda$ consists of a union of Schubert subvarieties $X_\mu\subset X_\lambda$ such that $\mu$ is a partition obtained from $\lambda$ obtained by adding a South-East hook to the Young diagram of $\lambda$. For a more precise version of this statement, we refer the reader to \cite{Coskun11} or to Section 9.3 of \cite{BL}.

\begin{ex*}
Suppose $n=8$, $d=3$. For the partition $\lambda=[4, 2, 1]$, $\Sing(X_\lambda)$ consists of two subvarieties $X_\mu$ and $X_{\mu'}$ with $\mu=[4, 3, 3]$ and $\mu'=[5, 5, 1]$. 
$$
\lambda= \ytableaushort
{\none, \none, \none}
*{5, 5, 5}
*[*(yellow)]{ 4,2, 1}\;\;\;\;\;\;\;  \mu=\ytableaushort
{\none, \none, \none}
*{5, 5, 5}
*[*(yellow)]{ 4, 2, 1}
*[*(red)] {4+0, 2+1, 1+2}\;\;\;\;\;\;\;  \mu'=\ytableaushort
{\none, \none, \none}
*{5, 5, 5}
*[*(yellow)]{ 4, 2, 1}
*[*(red)] {4+1, 2+3, 1+0}
$$
\end{ex*}

We define a set $S_\lambda:=\{(j, a) \mid (j, a)\in E_\lambda, \text{ $j<d$ and $j<a$}\}$. It can be easily shown that a pair $(j, a)$ is in $S_\lambda$ if and only if $(j+1, a)\in E_\mu$ for some $X_\mu\subseteq \Sing(X_\lambda)$. By this reasoning, we refer to $S_\lambda$ as the \emph{essential singular set} of $X_\lambda$. As before, we set $\overline{S_\lambda}:=\{a \mid (j, a)\in S_\lambda\}$. While ${E_\lambda}$ determines all corners of the Young diagram of $\lambda$, $S_\lambda$ picks up only the ``inside'' corners, which we show in the two following lemmas are those which define non-inclusion relations.

\begin{lem}\label{lem:inclusions1}
Let $X_\lambda$ be a Schubert subvariety of $\SB(d, \End_F(V))$ with
respect to the flag of right ideals $I_1\subset I_2\subset\cdots\subset
I_n=\End_F(V)$, with $\rrank(I_r)=r$. The following are equivalent:

\begin{enumerate}

	\item $\lambda_a$ defines an outside East corner of the Young diagram of $\lambda$, i.e. $n-d=\lambda_a>\lambda_{a+1}$

	\item $(a, a)\in E_\lambda$

	\item $a\leq d$ and $a\in \overline{E_\lambda}\setminus\overline{S_\lambda}$

	\item $I_a\subseteq J$ for all $J\in X_\lambda$ and there exists some $J'\in X_\lambda$ such that $I_{a+1}\not\subseteq J'$

	\end{enumerate} 
\end{lem}

\begin{proof}
	(1)$\iff$(2): It follows from the definition of $E_\lambda$ that $n-d=\lambda_a>\lambda_{a+1}$ if and only if $(a, n-d+a-\lambda_a)=(a, a)\in E_\lambda$. \\
	(2)$\implies$(3): It suffices to show that $a\notin \overline{S_\lambda}$, which follows immediately from the definition of $S_\lambda$, as $(j, a)\in S_\lambda\implies j<a$.\\
	(3)$\implies$(2): If $a\in\overline{E_\lambda}\setminus\overline{S_\lambda}$, then $(j, a)\in E_\lambda$ for some $j\leq d$ and either $j=d$ or $j=a$. If $a=d$, this implies immediately that $(j, a)=(a, a)\in E_\lambda$. On the other hand, if $a<d$ and $j=d$, then $\lambda_j=n-d+j-a>n-d$, a contradiction. So, we must have $j=a$ and hence $(a, a)\in E_\lambda$. \\
	(2)$\implies$(4): If $(a, a)\in E_\lambda$, then for all $J\in X_\lambda$
  we must have $\rrank(J\cap I_a)\geq a \implies I_a\subseteq J$. Suppose
  that for all $J\in X_\lambda$, $I_{a+1}\subseteq J$ or equivalently $\rrank(J\cap I_{a+1})\geq a+1$. This implies $\lambda_{a+1}=n-d=\lambda_a$, contradicting the assumption that $\lambda_a>\lambda_{a+1}$. \\
	(4)$\implies$(1): Suppose $\rrank(J\cap I_a)=a$ for all $J\in X_\lambda$
  but there exists some $J'\in X_\lambda$ such that $\rrank(J'\cap I_{a+1})<a+1$. It follows that $\lambda_a=n-d$ but $\lambda_{a+1}<n-d$.
\end{proof}

\begin{lem}\label{lem:inclusions2}
Let $X_\lambda$ be a Schubert subvariety of $\SB(d, \End_F(V))$ with
respect to the flag of right ideals $I_1\subset I_2\subset\dots\subset
I_n=\End_F(V)$, with $\rrank(I_r)=r$. The following are equivalent:

\begin{enumerate}

	\item $\lambda_d$ defines an outside South corner of the Young diagram of $\lambda$, i.e. $\lambda_d>0$

	\item $(d, a)\in E_\lambda$

	\item $a\geq d$ and $a\in \overline{E_\lambda}\setminus\overline{S_\lambda}$

	\item $I_{a}\supseteq J$ for all $J\in X_\lambda$ and there exists some $J'\in X_\lambda$ such that $I_{a-1}\not\supseteq J'$

\end{enumerate} 
\end{lem}

\begin{proof}

	(1)$\iff$(2): It follows from the definition of $E_\lambda$ that $\lambda_d>0$ if and only if $(d, a)\in E_\lambda$. \\
	(2)$\implies$(3): It suffices to show that $a\notin \overline{S_\lambda}$, which follows immediately from the definition of $S_\lambda$, as $(j, a)\in S_\lambda\implies j<d$.\\
	(3)$\implies$(2): If $a\in\overline{E_\lambda}\setminus\overline{S_\lambda}$, then $(j, a)\in E_\lambda$ for some $j\leq d$ and either $j=d$ or $j=a$. If $a=d$, this implies immediately that $(d, a)\in E_\lambda$. On the other hand, if $a>d$ and $j=a$, then $j>d$, a contradiction. So, we must have $j=d$ and hence $(d, a)\in E_\lambda$. \\
	(2)$\implies$(4): If $(d, a)\in E_\lambda$, then for all $J\in X_\lambda$
  we must have $\rrank(J\cap I_{a})\geq d \implies I_{a}\supseteq J$. Suppose that for all $J\in X_\lambda$, $I_{a-1}\supseteq J$. This implies $(d, a-1)\in E_\lambda$, a contradiction. \\
	(4)$\implies$(1): Suppose $\rrank(J\cap I_a)=d$ for all $J\in X_\lambda$.
  Since $\rrank(J\cap I_{n-\lambda_d})=d$ by definition of $X_\lambda$, it follows that $n-\lambda_d\leq a<n$ and hence $\lambda_d>0$. 
\end{proof}

\begin{ex*}
Consider the partition $\lambda=[4, 2, 1]$ for $n=8, d=3$ from the previous example. In this case, $E_\lambda=\{(1, 2), (2, 5), (3, 7)\}$, and $S_{\lambda}=\{(1, 2), (2, 5)\}$. 

$$
\lambda= \ytableaushort
{\none \none \none 2, \none 5, 7}
*{5, 5, 5}
*[*(yellow)]{ 4,2, 1}\;\;\;\;\;\;\;  \mu=\ytableaushort
{\none \none \none 2, \none, \none \none 5}
*{5, 5, 5}
*[*(yellow)]{ 4, 2, 1}
*[*(yellow)] {4+0, 2+1, 1+2}\;\;\;\;\;\;\;  \mu'=\ytableaushort
{\none, \none \none \none \none 2, 7}
*{5, 5, 5}
*[*(yellow)]{ 4, 2, 1}
*[*(yellow)] {4+1, 2+3, 1+0}
$$
\end{ex*}

With this, we obtain the following corollary to Theorem \ref{thm:smooth}, which can be seen as an extension of the result to the ``smooth'' intersection conditions of an arbitrary $X_\lambda$.

\begin{cor}\label{cor:smoothenough}
If $\SB(d, A)$ has a closed subvariety $P$ such that $P\otimes_F K\simeq X_\lambda$, then $\ind(A) \mid \gcd(\overline{E_\lambda} \setminus \overline{S_\lambda})$. Moreover, there exist right ideals $I_k, I_m\subseteq A$ such that for any $J\in P(L)$, $(I_k\otimes_F L)\subseteq J\subseteq (I_m\otimes_F L)$. 
\end{cor}

In order to provide the full converse to Proposition \ref{prop:easy} for an arbitrary Schubert variety $X_\lambda$, it remains to show that if $\SB(d, A)$ contains an $F$-form $P$ of $X_\lambda$, then $\ind(A)\mid \gcd(\overline{S_\lambda})$. To do this, we construct closed subvarieties of $P$ which are defined over $F$ and to which we can apply Corollary \ref{cor:smoothenough}. These subvarieties will be obtained from the structure of $\Sing(X_\lambda)$.

We rely on the fact for an $F$-variety $X$, the singular locus of $X_{\overline{F}}$ defines a Zariski-closed subset $Z$ of $X$. By equipping $Z$ with the reduced induced scheme structure, we may consider $Z$ as a subvariety of $X$ defined over $F$. In particular, if $\SB(d, A)$ contains a subvariety $P$ defined over $F$ such that $P_{K}\simeq X_{\lambda}$ for a splitting field $K/F$ of $A$, then $P$ has a closed subvariety $Z\subset P$ defined over $F$ such that $Z_{K}\simeq \Sing(X_\lambda)$.



%




\subsection{An iterative process}

In general, the singular locus of a Schubert variety may have many irreducible components, none of which are required to be smooth. The next step for dealing with such a variety is to iterate this process by considering ``the singular locus of a component of the singular locus" until we achieve subvarietes of $P$ which are $F$-forms of smooth (or ``smooth enough'') Schubert varieties.

Starting with a variety $P$ defined over $F$ with $P\otimes_F K\simeq X_\lambda$, we can iterate the subvariety construction to achieve a closed subvariety $Z\subset P$ defined over $F$ with $Z\otimes_F K\simeq X_\mu$, provided that $\mu$ can be obtained from $\lambda$ by adding a finite number of hooks. We begin with a technical lemma which provides a combinatorial description of some particular partitions which can be formed by adding hooks to a given partition $\lambda$.

\begin{lem}\label{lem:replacement}
Consider a partition $\lambda=[\lambda_1, \dots, \lambda_d]$ and suppose $\lambda_j$ corresponds to an inside corner. That is, $(j, a)\in E_\lambda$ with $j<d$ and $j<a$. 

\begin{enumerate}

		\item If $a\leq d$, adding $a-j$ hooks to $\lambda$ will result in a partition $\mu=[\mu_1, \dots, \mu_d]$ such that
		\begin{align*}
		\mu_i=\begin{cases}
						n-d		&\text{if $i\leq a$}\\
					\lambda_i	&\text{if $i>a$}
		\end{cases}
		\end{align*}

		\item If $a\geq d$, adding $d-j$ hooks to $\lambda$ will result in a partition $\mu=[\mu_1, \dots, \mu_d]$ such that
		\begin{align*}
		\mu_i=\begin{cases}
					\lambda_i		&\text{if $\lambda_i> n-a$}\\
					n-a				&\text{if $\lambda_i\leq n-a$}
		\end{cases}
		\end{align*}

		\item In both cases, $a\in \overline{E_\mu}\setminus \overline{S_\mu}$.

	\end{enumerate}
\end{lem}

\begin{proof}
We omit the proof of parts (1) and (2) as they are strictly computational. Consider $(j, a)\in E_\lambda$ and suppose first that $a\leq d$. Applying part (1), we have $(a, a)\in E_\mu$, and it follows by Lemma \ref{lem:inclusions1} that $a\in \overline{E_\mu}\setminus\overline{S_\mu}$. Next suppose $a\geq d$. Applying part (2) we have $(d, a)\in E_\mu$, and it follows by Lemma \ref{lem:inclusions2} that $a\in \overline{E_\mu}\setminus\overline{S_\mu}$.
\end{proof}

\begin{ex*}

	Consider $n=9$, $d=4$ and $\lambda=[4, 3, 1, 0]$. The diagram of $\lambda$ has 3 inside corners with $E_\lambda=\{(1, 2), (2, 4), (3, 7)\}$ and $\overline{S_\lambda}=\{2, 4, 7\}$. 
	$$
	\lambda= \ytableaushort
	{\none \none \none 2, \none \none 4, 7, \none}
	*{5, 5, 5, 5}
	*[*(yellow)]{4, 3, 1}
	$$

	\begin{itemize}

		\item For $(1, 2)$, we have $a=2<d$, so applying part 1 of Lemma \ref{lem:replacement}, we obtain the partition $\mu=[5, 5, 1, 0]$ after adding one hook. 

		\item For $(2, 4)$, we have $a=4=d$, so either part of Lemma \ref{lem:replacement} may be applied to obtain the partition $\alpha=[5, 5, 5, 5]$ after adding 2 hooks. 

		\item For $(3, 7)$, we have $a=7>d$, so applying part 2 of Lemma \ref{lem:replacement}, we obtain the partition $\beta=[4, 3, 2, 2]$ after adding one hook.

	\end{itemize}
		$$
		\mu= \ytableaushort
		{\none, \none \none \none \none 2, 7, \none}
		*{5, 5, 5, 5}
		*[*(yellow)]{4, 3, 1}
		*[*(red)]{4+1, 3+2}\;\;\;\;\;\;\;\;\;
			\alpha= \ytableaushort
			{\none, \none, \none, \none \none \none \none 4}
			*{5, 5, 5, 5}
			*[*(yellow)]{4, 3, 1}
			*[*(red)]{4+1, 3+2, 1+4, 5}\;\;\;\;\;\;\;\;\;
			\beta= \ytableaushort
			{\none \none \none 2, \none \none 4, \none, \none 7}
			*{5, 5, 5, 5}
			*[*(yellow)]{4, 3, 1}
			*[*(red)]{0, 0, 1+1, 2}
		$$

Note that $X_\alpha$ is smooth, while $S_\mu=\{(3, 7)\}$ and $S_\beta=\{(1, 2), (2, 4)\}$. 		
\end{ex*}


\subsection{Galois action on the singular locus}

We desire a stronger claim than the existence of an $F$-form of
$\Sing(X_\lambda)$. In particular, we would like to say that for any
Schubert variety $X_\mu\subseteq\Sing(X_\lambda)$, if $P$ is a twisted form
of $X_\lambda$ defined over $F$, then $P$ has a closed subvariety $Z\subset
P$, also defined over $F$, such that $Z$ is a twisted form of $X_\mu$.
More
precisely, we have:

\begin{lem}
Let $K/F$ be a Galois splitting field for $A$, and suppose that we have a
subvariety $P$ of $SB(d, A)$, such that $P_K = X_\lambda$. If $X_\mu
\subset \Sing(X_\lambda)$ is an irreducible component of the singular locus
of $X_\lambda$, defined by the addition of a hook to the Young diagram for
$\lambda$, then there exists a subvariety $Z \subset P$ such that $Z_K =
X_\mu$.
\end{lem}
\begin{proof}
By the geometric description of the irreducible components of the singular
set, it is automatic that the Galois action, which acts via elements of
$PGL(V_K)$ cannot nontrivially permute the components of the singular set.
Hence, considered as points on the Hilbert scheme of $SB(d, A)$, these
irreducible components are fixed by the Galois action, and hence correspond
to $F$-rational subvarieties $Z \subset P$ as claimed.
\end{proof}

\begin{prop}\label{prop:addinghooks}
Let $\lambda$ and $\mu$ be partitions defining Schubert subvarieties of $\Gr(d, n)$ such that $\mu$ is obtained from $\lambda$ by adding finitely many hooks. For a central simple $F$-algebra $A$ of degree $n$, if $\SB(d, A)$ contains a closed subvariety $P$ defined over $F$ such that $P_{\overline{F}}\simeq X_\lambda$, then $P$ contains a closed subvariety $Z$ defined over $F$ such that $Z_{\overline{F}}\simeq X_\mu$. 
\end{prop}

\begin{proof}
If $\mu$ is obtained from $\lambda$ by adding finitely many hooks, we may form a sequence $\alpha_1, \dots, \alpha_k$ such that $\lambda=\alpha_1$, $\mu=\alpha_k$ and for each $2\leq i\leq k$, $\alpha_{i}$ is obtained from $\alpha_{i-1}$ by adding precisely one hook. It follows from the definition of the singular locus that for each $2\leq i\leq k$, $X_{\alpha_i}\in \Sing(X_{\alpha_{i-1}})$.

Under the assumption that $\SB(d, A)$ contains a twisted form of $X_{\alpha_1}$ over $F$, the above argument implies that $\SB(d, A)$ must also contain a twisted form of $X_{\alpha_2}$ over $F$. By induction on $i$, we obtain the result that $\SB(d, A)$ must finally contain a twisted form of $X_{\alpha_k}=X_\mu$ defined over $F$. 
\end{proof}

This process yields the desired converse to Proposition \ref{prop:easy}, and can in fact be taken one step further to show that such a variety is in fact defined by ``Schubert-style'' intersection conditions as in \eqref{intersections}. The result can therefore be seen as a complete generalization of Theorems \ref{thm:Artin} and \ref{thm:subgrassmannian} to the set of all twisted Schubert subvarieties of $\SB(d, A)$.

\begin{thm} \label{thm:all}
The generalized Severi-Brauer variety $\SB(d, A)$
has a closed subvariety $P$ such that $P\otimes_F K\simeq X_{\lambda}$
for a Schubert subvariety $X_\lambda$ 
if
and only if $\ind(A)\mid \gcd(\overline{E_\lambda})$. Moreover, in this
case, $A$
contains a flag of right ideals $I_{a_1}\subset \cdots \subset I_{a_r}$ for
$\overline{E_{\lambda}}=\{a_1, \dots, a_r\}$ such that for any finite
extension $L/F$, 
$$
P(L)=\{J\subseteq A_L: \rk(J\cap (I_a)_L)\ge j \mbox{ for } (j,a)\in
E_{\lambda}\}
$$
\end{thm}

\begin{proof}
If $X_\lambda$ is smooth, then the result follows immediately from Theorem \ref{thm:smooth}. 
Suppose $(j, a)\in {S_\lambda}$ for some $j<d$ and $j<a$. Using the replacement process described in Lemma \ref{lem:replacement} together with Proposition \ref{prop:addinghooks}, $P$ has a closed subvariety $Z$ defined over $F$ such that $Z_{\overline{F}}\simeq X_\mu$ where $\mu$ is obtained from $\lambda$ by adding hooks and $a\in \overline{E_\mu} \setminus \overline{S_\mu}$. Applying Corollary \ref{cor:smoothenough} we must have $\ind(A) \mid a$.

Now, suppose $\ind(A)\mid \gcd(\overline{E_\lambda})$. This condition
implies that $A$ contains a flag of right ideals $I_{a_1}\subset \cdots
I_{a_r}$ for  $\overline{E_{\lambda}}=\{a_1, \dots, a_r\}$. For a splitting
field $K/F$ of $A$, fix a full flag of right ideals
${I_1}'\subset{I_2}'\subset\cdots\subset{I_n}'=A_K$ such that $I_{a_j}\otimes_F K={I_{a_j}}'$ for all $a_j\in \overline{E_\lambda}$. Let $X_\lambda$ be the Schubert subvariety of $\SB(d, A_K)$ defined by $\lambda$ with respect to this flag.

Denote by $P_\lambda$ the $F$-subvariety of $\SB(d, A)$ defined as in the statement of the theorem. That is, for any $L/F$, 
$$
P_\lambda(L):=\{J\subseteq A_L: \rk(J\cap (I_a)_L)\ge j \mbox{ for } (j,a)\in E_{\lambda}\}
$$

If $P$ is a twisted form of $X_{\lambda}$ defined over $F$, the goal is to show that $P=P_\lambda$. 

Let $L/F$ be an arbitrary finite field extension and let $J\in P(L)$.  After extending to a splitting field $K/L$, we find that for any $(j,a)\in E_{\lambda}$, we have 
$$
\rk(J\cap (I_a)_L)=\rk((J\cap (I_a)_L)_K)=\rk(J_K\cap {I_a}')\ge j
$$ 
since $I_K\in P_K(K)=X_{\lambda}(K)$. So $P(L)\subseteq P_{\lambda}(L)$.

We have that $i:P\hookrightarrow P_{\lambda}$ is an inclusion of $F$-varieties since $i_L:P(L)\hookrightarrow P_{\lambda}(L)$ for all finite field extensions $L/F$. If $K/F$ is a splitting field for $P$, then 
$i_K$ induces the identity map.  So $\coker(i)_K=\coker(i_K)=0$.  Thus, $\coker(i)$ is a form of the zero variety and so $\coker(i)=0$.  It follows that $P=P_{\lambda}$ as required. 
\end{proof}


\section{Schubert cycles}

Schubert subvarieties are of particular interest as they form the building
blocks for the Grothendieck group and Chow group of $\Gr(d, n)$. In this
section we apply our results on twisted forms of Schubert varieties to
questions concerning rational cycles and torsion elements in these
cohomology theories.


\subsection{The topological filtration}

Let $X$ be a smooth projective variety over $F$ and consider the Grothendieck group 
$$
\K0(X)=\langle [\mathcal{O}_V] \mid V\subseteq X \text{ closed subvariety}\rangle.
$$

The topological filtration on $\K0(X)$ can be defined by setting
$$
\K0(X)^{(i)}:=\langle [\mathcal{O}_V] \mid \codim(V)\geq i\rangle \;\;\;\;\text{ and then }\;\;\;\; T^{i}(X):=\K0(X)^{(i)}/\K0(X)^{(i+1)}
$$
Note that $T^1(X)\cong \CH^1(X)=\Pic(X)$ and $T^2(X)\cong \CH^2(X)$. In general, there is a natural surjection $\CH^k(X)\to T^k(X)$ which only induces an isomorphism over $\ZZ[\frac{1}{(k-1)!}]$.

We denote by $\overline{X}$ the variety $X$ over the algebraic closure of $F$. For the case $\overline{X}=\Gr(d,n)$,  a $\ZZ$ basis of $T^i(\Gr(d,n))$ is given by 
$$
\{\Sigma_{\lambda}: |\lambda|=i, n-d\ge \lambda_1\ge \cdots \ge \lambda_d\ge 0\}
$$
where $\lambda=[\lambda_1, \dots, \lambda_d]$ is a partition of $|\lambda|=\sum_{j=1}^d\lambda_j$ and $\Sigma_\lambda=[\mathcal{O}_V]$ for $V=X_\lambda$.

For $X=\SB(d, A)$ and a splitting field $K/F$ for $A$, $\K0(X)$ does not have such a basis, but we may consider the image of the restriction map 
$$\res^i:T^i(X)\to T^i(X_{K})\simeq T^i(\Gr(d, n)).$$

In particular, we can ask: \emph{What is the smallest integer $a_\lambda\in\ZZ$ such that $a_\lambda\Sigma_\lambda\in \im(\res^i)$?} An initial upper bound for these integers in the usual Severi-Brauer variety case was provided by Karpenko.

\begin{thm}[\cite{Ka95} Lem. 3]
Let $X=\SB(A)$ with $\ind(A)=r$. For any $0\leq i\leq \dim(X)$,  
$$
\left(\dfrac{r}{\gcd(r, i)}\right) \Sigma_{i}\in \im\left(T^i(X)\xrightarrow{\res^i}T^i(\overline{X})\right)
$$
\end{thm}

Replacing codimension $i$ by the $\gcd$ of the projection of the essential set $\overline{E_\lambda}$ for a partition $\lambda$ with $|\lambda|=i$, we obtain a similar expression for generalized Severi-Brauer varieties.

\begin{thm}\label{thm:cycles}
Let $X=\SB(d, A)$ with $\ind(A)=r$. For any $0\leq i\leq \dim(X)$ and $|\lambda|=i$,
$$
\left(\frac{r}{\gcd(r, \overline{E_\lambda})}\right) \Sigma_{\lambda}\in  \im\left(T^i(X)\xrightarrow{\res^i}T^i(\overline{X})\right)
$$
\end{thm}

\begin{proof}
This proof is a direct generalization of Karpenko's proof in \cite{Ka95}. Without loss of generality we may assume $F$ has no extension of degree prime to $p$ for some prime $p$. Under this assumption, suppose $r=p^a$ for some $a$ and $\gcd(r, \overline{E_\lambda})=p^b$ for some $b\leq a$. Let $E/F$ be a field extension of degree $p^{a-b}$ such that $\ind(A_E)=p^b$. Since $p^b\mid \gcd(\overline{E_\lambda})$, $X_E$ contains a closed subvariety $P$ defined over $E$ such that $P_K\simeq X_\lambda$ for a splitting field $K$ of $A$. So, $[P]\mapsto \Sigma_\lambda$ under the restriction map $T^i(X_E)\to T^i(\overline{X})$.

Applying the norm map $N_{E/F}$ to this element, we obtain 
$$
\left(\frac{r}{\gcd(r, \overline{E_\lambda})}\right)\Sigma_\lambda=p^{a-b}\Sigma_\lambda\in \im\left(T^i(X)\xrightarrow{\res^i} T^i(\overline{X})\right)
$$
\end{proof}

The Chow group $\CH(\overline{X})$ is generated by Schubert cycles $\sigma_\lambda=[X_\lambda]$. Replacing $\res^i$ by the restriction map $\CH^i(X)\to \CH^i(\overline{X})$, we obtain the following corollary.

\begin{cor}
Let $X=\SB(d, A)$ with $\ind(A)=r$. For any $0\leq i\leq \dim(X)$ and $|\lambda|=i$,
$$
\left(\frac{r}{\gcd(r, \overline{E_\lambda})}\right) \sigma_{\lambda}\in  \im\left(\CH^i(X)\to \CH^i(\overline{X})\right)
$$
\end{cor}

These rational bundles (resp. cycles) arise from our study of closed subvarieties of $\SB(d, A)$ in the previous section. In general, however, these bundles (resp. cycles) are not generators of $\im(\res^i)$. We illustrate this fact by considering the case of codimension $1$.


\subsection{The Picard group}

For any choices of $d$ and $n$, we have $T^1(\Gr(d, n))=\Z\Sigma_{1}$ and $\overline{E_{1}}=\{n-d\}$, where $1$ denotes the partition $[1, 0, \dots, 0]$. Along with the property that $\ind(A)\mid n$, it follows from Theorem \ref{thm:cycles} that 
$$
\left(\dfrac{\ind(A)}{\gcd(\ind(A), d)}\right)\Z\Sigma_1\subseteq  \im(\res^1)
$$

Recall that $\exp(A)|\ind(A)$ and $\exp(A)$ and $\ind(A)$ share prime divisors. We have equality in the above expression if and only if $\ind(A)=\exp(A)$.

\begin{prop}\label{prop:picard}
Let $X=\SB(d, A)$. The image of the restriction map $\res^1:T^1(X)\to T^1(\overline{X})$ is generated by  
$$
\left(\frac{\exp(A)}{\gcd(\exp(A), d)}\right)\Sigma_{1}
$$
\end{prop}

\begin{proof}
Recall that $T^1(X)\cong\Pic(X)$. For a projective homogeneous variety $X=G/P_{\Theta}$ of type $\Theta\subset \Sigma$, where $\Sigma$ is the set of simple roots of the root system of a simple algebraic group $G$ with respect to a maximal torus $T$, $\Pic(\overline{X})$ is a free Abelian group on the line bundles corresponding to the fundamental dominant weights of $\Pi\setminus \Sigma$. There exists an exact sequence (cf. \cite[Section 2.1]{MT})
$$
0\to \Pic(X)\xrightarrow{\res^1} \Pic(\overline{X})\xrightarrow{\alpha_X} \Br(F)
$$
where the map $\alpha_X$ is defined on the line bundles corresponding to fundamental weights of $\Sigma \setminus \Theta$, $\alpha_X(\LL(\omega_i))= [A_{\omega_i}]$. Here $A_{\lambda}$ is the Tits algebra corresponding to the dominant weight $\lambda$.

Now consider the special case of $X=\SB(d,A)$ for $G=\PGL(A)$. $\Pic(\overline{X})$ is a free Abelian group on $\LL(\omega_d)$ where $\omega_i, i=1,\dots,n-1$ are the fundamental dominant weights of $\rA_{n-1}$. In the above sequence, $\alpha_X(\LL(\omega_d))= [A_{\omega_d}]=[A^{\otimes d}]$~\cite[2.4.1]{MT}. Considering the isomorphism $T^1(X)\cong \Pic(X)$, we see that $\Sigma_1\simeq \LL(\omega_d)\in \im(\res^1)$ if and only if $\exp(A^{\otimes d})=1$.
\end{proof}


\subsection{Littlewood-Richardson rules}

Using the ring structure of $\K0(\Gr(d, n))$ via Littlewood-Richardson rules for multiplying Schubert cycles we can refine these rational elements even further.

The multiplicative structure of $\CH(\Gr(d, n))$ is determined via a multiplicative law on the Schubert cycles, which is in turn defined by the so-called the Littlewood-Richardson rules. The Grothendieck group $\K0(\Gr(d, n))$ also has a ring structure defined by a similar set of rules, which are presented in \cite{Bu02}.

For partitions $\lambda$ and $\mu$, the products in $\CH$ and $\K0$ of their corresponding Schubert cycles/bundles are given by the following rules:
$$
\sigma_\lambda\sigma_\mu=\sum_{|\nu|=|\lambda|+|\mu|} s^\nu_{\lambda\mu}\sigma_\nu \in \CH(\Gr(d, n))\;\;\;\;\text{ and }\;\;\;\;\Sigma_\lambda\Sigma_\mu=\sum_{|\nu|\geq|\lambda|+|\mu|} r^\nu_{\lambda\mu}\Sigma_\nu\in \K0(\Gr(d, n))
$$

for Littlewood-Richardson coefficients $s^\nu_{\lambda\mu}, r^\nu_{\lambda\mu}\in \ZZ$ such that $r_{\lambda\mu}^\nu=s_{\lambda\mu}^\nu$ whenever $|\nu|=|\lambda|+|\mu|$.

Applying the topological filtration to the right-hand formula will remove all terms corresponding to the case $|\nu|>|\lambda|+|\mu|$, thus reducing to the original Littlewood-Richardson rule given by the left-hand formula. With justification given by this equivalence of rules under filtration, we abuse notation by denoting the equivalence class of $\Sigma_\lambda$ in $T^i(X)$ by $\sigma_\lambda$, and refer to these classes as Schubert cycles.

\begin{ex*}
On $\Gr(2, 4)$, the multiplicative structure on Schubert cycles is given by 
\begin{align*}
	\sigma_1^2&=\sigma_{1, 1}+\sigma_2\\
	\sigma_1\sigma_{1, 1}&=\sigma_1\sigma_2=\sigma_{2, 1}\\
	\sigma_{1, 1}^2&=\sigma_2^2=\sigma_{2, 2}\\
	\sigma_{1, 1}\sigma_2&=0
\end{align*}

\end{ex*}

Using these relations, we are able to improve the results concerning rational cycles on $\SB(2, A)$ for $\deg(A)=4$. As before, we consider the image of the map $\res^i:T^i(\SB(d, A))\to T^i(\Gr(d, n)).$

This proposition improves on the results of Theorem~\ref{thm:cycles} in the case of $\SB(2,A)$ for $\deg(A)=4$ by using Proposition~\ref{prop:picard} and the above relations.

\begin{prop} \label{prop:gr24}
Let $X=\SB(2, A)$ with $\deg(A)=4$. The image of the restriction map $\res^i:T^i(X)\to T^i(\overline{X})$ satisfies the following relations
\begin{align}
\im(\res^1)&=\left(\frac{\exp(A)}{\gcd(2, \exp(A))}\right)\ZZ\sigma_1 \label{codim1}\\
\im(\res^2)&\supseteq \left( \frac{\exp(A)}{\gcd(2, \exp(A))}\right)\ZZ\sigma_1^2 + \ind(A)\ZZ\sigma_{1, 1} \label{codim2}\\
\im(\res^3)&\supseteq \exp(A)\ZZ\sigma_{2, 1} \label{codim3}\\
\im(\res^4)&\supseteq\left(\frac{\ind(A)}{\gcd(2, \ind(A))}\right)\ZZ\sigma_{2, 2} \label{codim4}
\end{align}
\end{prop}

\begin{proof}
It suffices to consider the cases of $\exp(A)=2$ and $\exp(A)=4$ as the split case is trivial.

Line \eqref{codim1} follows from Proposition \ref{prop:picard}.

For line \eqref{codim2}, we combine \eqref{codim1} with the product $\sigma_1^2=\sigma_{1, 1}+\sigma_2$: If $\exp(A)=2$, then $\sigma_1\in \im(\res^1)$, and hence $\sigma_1^2\in \im(\res^2)$. On the other hand, if $\exp(A)=4$, there exists a quadratic extension $E/F$ such that $\exp(A_E)=\ind(A_E)=2$. So, $\sigma_1\in \im(T^1(\SB(2, A))\to T^1(\SB(2, A_E)))$ and it follows that $\sigma_1^2\in \im(T^2(\SB(2, A))\to T^2(\SB(2, A_E)))$. Using the norm map $N_{E/F}$, we then obtain $2\sigma_1^2\in \im(\res^2)$, as desired.



For line \eqref{codim3}, we use the relation $\sigma_1^3=2\sigma_{2, 1}$ together with \eqref{codim1}: If $\exp(A)=2$, then $\sigma_1\in \im(\res^1)$ and hence $\sigma_1^3=2\sigma_{2, 1}\in \im(\res^3)$. If $\exp(A)=4$, then as above, we consider a quadratic field extension $E/F$ such that $\exp(A_E)=2$. Then, $\sigma_1^3\in\im(T^3(\SB(d, A))\to T^3(\SB(2, A_E))$ and applying the norm map $N_{E/F}$, we obtain the result.  

 Line \eqref{codim4} follows directly from Theorem \ref{thm:cycles} after observing that $E_{2,2}=\{(2,2)\}$.
\end{proof}


\subsection{Bounds for torsion in the topological filtration}

The calculations done in the previous section provide improved upper bounds for the coefficients $a_\lambda$ such that $a_\lambda\sigma_\lambda\in \im(\res^i)$. Using the following proposition, we can compare these upper bounds with lower bounds given by the computation of $\K0(X)$ by Quillen \cite{Qu73}. In some cases


 these bounds coincide and thus provide generators for $\im(\res^i)$.

\begin{prop}[\cite{Ka95} Prop. 1]\label{prop:karpenko}
Let $X$ be a variety such that $\K0(X)\to \K0(\overline{X})$ is injective. If $|T^{*}(\overline{X})/\im(T^{*}(X))|$ is finite, then 
$$
|\Tors(T^*(X))|=\frac{|T^{*}(\overline{X})/\im(T^{*}(X))|}{|\K0(\overline{X})/\K0(X)|}
$$
\end{prop}

In the case that $A$ is of prime index $p$ and $X=\SB(d, A)$, for $1\le d\le p$, this allows us to produce an upper bound on $|\Tors(T^*(X))|$ in terms of the combinatorial data provided by $E_\lambda$ for each partition $\lambda$.

\begin{cor}\label{cor:prime}
Let $X=\SB(d, A)$ with $\ind(A)=p$ for $p$ prime. Define the sets $M_X=\{\lambda \mid p\nmid \gcd(\overline{E_\lambda})\}$ and $N_X=\{\lambda \mid p\nmid |\lambda|\}$. Then, 
$$
|\Tors(T^*(X))|\mid p^{|M_X|-|N_X|}
$$

In particular, if $\ind(A)=\deg(A)=p$ for $p$ prime, we have
$$
|\Tors(T^*(X))|\mid p^{|S(X)|}
$$
where $S_X=\{\lambda \mid p\mid |\lambda|, |\lambda|>0\}$.
\end{cor}

\begin{proof}
It suffices to show that $|T^*(\overline{X})/\im(T^*(X))|\mid p^{|M_X|}$ and that $|\K0(\overline{X})/\K0(X)|= p^{|N_X |}$.

Suppose $|\lambda|=i$. From Theorem \ref{thm:cycles}, we know that $\left(\frac{p}{\gcd(p, \overline{E_\lambda})}\right)\Sigma_\lambda\in \im(\res^i)$. It follows that  $p\Sigma_\lambda\in\im(\res^i)$ for all $\lambda$ and that $\Sigma_\lambda\in\im(\res^i)$ if $p\mid \overline{E_\lambda}$. Taking the product over all $1\leq i\leq \dim(X)$ yields the first result.

By Quillen's computation of $\K0(X)$, we have	
$$
|\K0(\overline{X})/\K0(X)|=\prod_{\lambda}\ind(A^{\otimes|\lambda|})
$$ 
Since $\exp(A)=p$, the algebra $A^{\otimes |\lambda|}$ is split if $p\mid |\lambda|$ and has index $p$ otherwise.

The second statement follows from the fact that $\overline{E_{\lambda}}\subset \{1,\dots,p-1\}$ for any partition $\lambda$ with $|\lambda|>0$. So if $\deg(A)=p$, $M_X$ consists of all non-zero partitions in a $d$ by $(p-d)$ box.  For $d=1$, this recovers the known result that $\CH^2(\SB(A))$ is torsion-free for $\deg(A)=p$ since no partitions in a $1\times(p-1)$ box could be of size divisible by $p$.
\end{proof}



In practice however, these bounds are far from sharp. To improve the bound for the case of $X=\SB(2, A)$ with $\deg(A)=4$, we consider the following construction (as in \cite{Kr10}).

Consider the map $\Gr(2, 4)\to \Pj^5$ given by the Pl\"{u}cker embedding. Fixing a 4-dimensional vector space $V$, we can think of this map as taking a 2-dimensional subspace $W\subset V$ to the 1-dimensional subspace $\bigwedge^2W\subset \bigwedge^2V$. Under this morphism, $\Gr(2, 4)$ becomes isomorphic to a quadric hypersurface in $\Pj^5$ associated to a bilinear form on $\bigwedge^2V$. The Pl\"{u}cker embedding is $\PGL(V)$-invariant, so for a central simple algebra $A$ with $\deg(A)=4$ (corresponding to a cocycle $\alpha\in H^1(F, \PGL_4)$), we obtain a morphism $\SB(2, A)\to \SB(B)$, where the cocycle corresponding to $B$ is given by composing $\alpha$ with the standard representation $\PGL(V)\to \PGL(V\wedge V)$. Now, $\deg(B)=6$, and by a result of Artin~\cite[4.4,4.5]{Ar82}, $[B]=[A^{\otimes 2}]\in\Br(F)$. Note then that $\ind(B)=\ind(A^{\otimes 2})$.

\begin{prop}\label{prop:plucker} 
If $X=\SB(2, A)$ with $\deg(A)=4$, $T^*(X)$ is torsion-free.
\end{prop}

\begin{proof}
For a generator $[H]\in T^1(\Pj^5)$, the cycle $\left(\frac{\ind(B)}{\gcd(2, \ind(B))}\right)[H]^2$ is rational in $T^2(\SB(B))$ by Proposition~\ref{prop:karpenko}. The pullback map $T^2(\SB(B))\to T^2(\SB(2, A))$ sends this cycle to $\left(\frac{\ind(B)}{\gcd(2, \ind(B))}\right)\sigma_1^2$.

Recall that $\ind(B)=\ind(A^{\otimes 2})$ and
\begin{equation} 
\ind(A^{\otimes r})\leq \ind(A) \mbox{ with equality if and only if }\gcd(\ind(A),r)=1 \label{eq:indpowers}
\end{equation}
Since we have $\ind(A)\mid 4$, we must have $\ind(A^{\otimes 2})\mid 2$ and so we find that
\begin{equation}
\sigma_1^2=\sigma_{(1,1)}+\sigma_2\in \im(\res^2),\label{eq:imprres2}
\end{equation}
which improves the result of Proposition \ref{prop:gr24}.

In particular, we obtain
$$
\im(\res^2)\supseteq \ZZ\sigma_1^2+\ind(A)\ZZ_{\sigma_{1, 1}}.
$$

In order to use Proposition \ref{prop:karpenko} to show that $\Tors(T^*(X))$ is trivial, we need to show that 
$$
|T^*(\overline{X})/\im(T^*(X))|=|\K0(\overline{X})/\K0(X)|
$$ in all cases. By (\ref{eq:indpowers}) and $\exp(A)|4$, we obtain 
\begin{equation}
|\K0(\overline{X})/\K0(X)|=(\ind(A))^2(\ind(A^{\otimes 2}))^2 \label{eq:ko}
\end{equation}
in all cases. Since the result is trivial in the split case, we need only consider the cases 
$$
(\ind(A),\ind(A^{\otimes 2}))\in \{(2,1),(4,1),(4,2)\}
$$ 
by equation (\ref{eq:indpowers}). In the first 2 cases, $\exp(A)=2$. So by (\ref{eq:ko}), $|\K0(\overline{X}/\K0(X)|=(\ind(A))^2$. Using Proposition \ref{prop:gr24}, we see that
$$
|T^*(\overline{X})/\im(T^*(X))|=\ind(A)\exp(A)\frac{\ind(A)}{2}=\ind(A)^2
$$ 
as required.

In the case $(\ind(A),\ind(A^{\otimes 2}))=(4,2)$, we have $\exp(A)=4$.

Using Proposition \ref{prop:gr24} and (\ref{eq:imprres2}), we see that
$$
|T^*(\overline{X})/\im(T^*(X))|=\frac{\exp(A)}{2}\ind(A)\exp(A)\frac{\ind(A)}
{2}=4^22^2=\ind(A)^2\ind(A^{\otimes 2})^2
$$ 
as required.
\end{proof}

Since $\CH^2(X)$ is canonically isomorphic to $T^2(X)$, this implies that $\CH^2(X)$ is also torsion-free for $X=\SB(2, A)$ with $\deg(A)=4$. In the case that $\ind(A)=2$, this recovers a previously-known result for quadrics.


\section{Motivic Decompositions}

In this section, we will apply Brosnan's results from \cite{Br05} on
decompositions of Chow motives of projective homogeneous varieties for
isotropic reductive groups and Proposition \ref{prop:plucker} to show that
$\CH^2(\SB(2,A))$ is torsion-free in the case that $A$ is a central simple
$F$-algebra of index dividing 12.

Throughout, for a smooth projective variety $X/F$, we write $M(X)$ for the
Chow motive of $X$, as defined in \cite{Man68}.

Fix a projective homogeneous variety $X$ for an isotropic reductive group $G$ and a cocharacter $\lambda:\Gm\to T$ where $T$ is a maximal torus of the group $G$.

Brosnan gives an explicit decomposition of the Chow motive of $X$ into a
coproduct of twisted Chow motive of projective quasi-homogeneous schemes for the centraliser $Z(\lambda)$ of the cocharacter $\lambda$.

He then obtains a decomposition of the form 
$$
M(X)=\oplus_i M(Z_i)(a_i)
$$
where $Z_i$ are projective quasi-homogeneous schemes for the centraliser $Z(\lambda)$. In fact the $Z_i$ are the irreducible components of the fixed points of $X$ under $\lambda$.

In terms of codimension $d$ Chow groups, this implies
$$
\CH^d(X)=\oplus_i \CH^{d+a_i}(Z_i)
$$

We decompose the Chow motive of $\SB(2,A)$ into twisted Chow motives of smaller generalized Severi-Brauer varieties using Brosnan's result. We then use this decomposition to show that $\CH^2(\SB(2,A))$ is torsion-free for $A$ of index dividing 12.

We first recall Brosnan's result: Theorem 7.1 of \cite{Br05} gives a motivic decomposition for a projective homogeneous variety of an isotropic reductive group $G$ with respect to a cocharacter $\lambda:\Gm\to T$ of a fixed maximal torus $T$ of $G$.  The projective homogeneous variety $X$ can be expressed as $X=G/P_J$ where $J$ is a $*$-invariant subset of the simple roots $\Sigma$ of $G$ with respect to $T$.  Here $*$ refers to the $*$-action of the absolute Galois group of $F$ on $\Sigma$. It is also assumed that the parabolic subgroup $P(\lambda)$ corresponding to the cocharacter $\lambda$ is defined over the field $F$ and so is of the form $P_I$ where $I$ is a $*$-invariant subset of $\Sigma$ containing $\Sigma_0$, which is the complement in $\Sigma$ of the set of distinguished (circled) vertices corresponding to the Tits index of the group $G$. In the special case of an isotropic reductive group of inner type, the $*$-action is trivial, and the statement can be simplified.

\begin{thm}[Special case of Theorem 7.1 in \cite{Br05} for groups of inner type .]\label{thm:brosnan}

Let $X$ be a projective homogeneous variety for an isotropic reductive group $G$ of inner type. Assume $X$ is a twisted form of $G/P_J$ for some subset $J$ of the simple roots $\Sigma$ with respect to $G$ and some maximal torus $T$. Let $\lambda:\Gm\to T$ be a cocharacter of $G$ such that the associated parabolic subgroup $P(\lambda)=P_I$ where $I$ is a subset of simple roots $\Sigma$ containing $\Sigma_0$ where $\Sigma_0$ is the complement in $\Sigma$ of the set of distinguished (circled)  vertices corresponding to the Tits index of the group $G$. Let $W=W(G,T)$ be the Weyl group with respect to $\Sigma$. Then the motive of $X$ can be decomposed as:
$$
M(X)=\oplus_{w\in E} M(Z_w)(\ell(w))
$$
where $E=W_I\backslash W / W_J$ is the set of minimal length double coset representatives of the Weyl group $W=W(G,T)$ with respect to the corresponding parabolic subgroups $W_I$ and $W_J$ and $\ell(w)$ is the length of $w$ with respect to reflections in simple roots $\Sigma$. Here $Z_{w}$ is a projective homogeneous variety which is a twisted form of $L_I/P_{J_w}$ over $k_s$. The set $J_w$ is determined by 
$$
J_w=\{\alpha\in I: w^{-1}\alpha\in R_J\}
$$
where $R_J$ is the set of roots with base $J$.
\end{thm}

We apply this theorem to our situation: Let $A$ be a central simple $F$-algebra of index $d$ and degree $n=kd$, where $d,k\ge 2$. The Tits index of the group of inner type $G=\PGL(A)$ is the graph $\Sigma=\{\alpha_1,\dots,\alpha_{n-1}\}$ consisting of the simple roots of the root system $A_{n-1}$ with 
$$
\Sigma_0=\Sigma-\{\alpha_{m}: d|m\}
$$
the complement of the set of circled vertices, where $d$ is the index of
$A$. By construction, $A=M_k(D)$ for a division algebra $D$ over $F$ of
degree $d$; in fact, $D=\End_A(\fA)$ where $\fA$ is a right ideal of $A$ of
reduced rank $d$. Let $T$ be a maximal torus of $\PGL(A)$ containing
a maximal torus of  $\PGL(D)\times \PGL(M_{k-1}(D))$.  Let $\lambda:\Gm\to
T$ be the cocharacter $t\mapsto
\overline{\diag(t\id_{D},\id_{M_{k-1}(D)})}\in \PGL(D)\times
\PGL(M_{k-1}(D))$. One could also identify $M_{k-1}(D)=\End_A(\fA^{k-1})$
and claim that this decomposition comes from expressing $A$ as a  direct
sum of ideals $A=\fA\oplus (\fA)^{k-1}$. The centraliser $Z(\lambda)$ of
$\lambda$ is then isomorphic to $\PGL(D)\times \PGL(M_{k-1}(D))$.
$P(\lambda)$ is the isotropic parabolic subgroup $P_I$ where
$I=\Sigma-\{\alpha_d\}$ contains $\Sigma_0$. The parabolic subgroup
defining $\SB(2,A)$ corresponds to the set of simple roots
$J=\Sigma-\{\alpha_2\}$. The Weyl group is $W=S_n$ and the corresponding
parabolic subgroups are $W_I=S_d\times S_{n-d}$ and $W_J=S_2\times
S_{n-2}$. To find the components of the motivic decomposition of
$M(\SB(2,A))$, we first need to find the unique minimal length
representatives of the set of double cosets $S_d\times S_{n-d}\backslash
S_n/S_2\times S_{n-2}$.


\subsection{Minimal Length Double Coset Representatives for Young subgroups of Symmetric Groups}

Determining the minimal length double coset representatives for Young subgroups of symmetric groups is a classical combinatorial problem. We recall the following definitions on Young tableaux:

\begin{defn}
For partitions $\mu,\lambda$ of $n$, a $\mu$-tableau of type $\lambda$ is a Young tableau of shape $\mu$ with $\lambda_i$ copies of $i$ for $i=1,\dots,k$ where $\lambda=(\lambda_1,\dots,\lambda_k)$.  A $\mu$-tableau of type $\lambda$ is row-semistandard if and only if the row entries are non-decreasing. The standard $\mu$-tableau of type $\lambda$ has  non-decreasing entries starting from the top left entry and proceeding along rows from left to right and then columns from top to bottom.  A labeled $\mu$-tableau of type $\lambda$ has labeled entries: for each $i=1,\dots,k$, there are entries $i_m$, for $m=1,\dots,\lambda_i$.  The subscripts for a given $i$ must increase from left to right along rows and then top to bottom along columns. The labeled standard $\mu$-tableau of type $\lambda$ would then be filled with entries $1_m: 1\le m\le \lambda_1; \dots; k_m: 1\le m\le \lambda_k$ filling each row from left to right, and then columns from top to bottom.
\end{defn}

According to an expository reference~\cite{Wil}, the double cosets of $S_{\lambda}\backslash S_n /S_{\mu}$  are in bijection with the row semistandard $\mu$-tableau of type $\lambda$ where $\lambda$ and $\mu$ are partitions of $n$, where the Young subgroup $S_{\lambda}$ for a partition $\lambda=(\lambda_1,\dots,\lambda_k)$ of $n$ is 
$$
S_{\lambda}=S_{\{1,\dots,m_1\}}\times \dots \times S_{\{m_{k-1}+1,\dots, m_k\}},
$$
where $m_i=\sum_{j=1}^i\lambda_j$ for all $1\le i\le k$. (Since $\lambda$ is a partition of $n$, $m_k=n$). The unique representative of minimal length of a double coset corresponding to a labeled row-semistandard $\mu$ tableau $t$ of type $\lambda$, is the element of $S_n$ which sends the labeled standard $\mu$-tableau of type $\lambda$ to $t$.

\begin{prop} 
Let $d\ge 2, n\ge 2d$. In terms of simple reflections $s_i=(i\;\; i+1)$, $i=1, \dots, n-1$, the minimal length double coset representatives of 
$$
S_{\{1,\dots,d\}}\times S_{\{d+1,\dots,n\}}\backslash S_n / S_{\{1,2\}}\times S_{\{3,\dots,n\}}
$$ 
are $w_0=\id,w_1=s_ds_{d-1}\cdots s_2$, and $w_2=s_ds_{d-1}\cdots s_1s_{d+1}s_d\cdots s_2$.
\end{prop}

\begin{proof} 
We will apply the results of \cite{Wil} stated above with a few important warnings. One is that the results of \cite{Wil} are set up for Young subgroups based on partitions (i.e. with non-increasing parts). Our Young subgroups are flipped by the assumptions on $n$ and $d$ and have non-decreasing parts.  Also, the results in \cite{Wil} are with respect to a right action of a symmetric group. Our application assumes a left action and our answer will reflect that choice.

To adapt to the point of view of \cite{Wil}, we will first find the minimal length double coset representatives of $S_{\lambda}\backslash S_n / S_{\mu}$ where $\lambda=(n-d,d)$ and $\mu=(n-2,2)$ with respect to a right action of $S_n$ and will then adjust accordingly.

Note that the standard $\mu$-tableau of type $\lambda$ would have 2 rows, of lengths $n-2$ and $2$ respectively, and would be filled with entries $1_1,\dots,1_{n-d}$ and then $2_1,\dots,2_d$.
$$
\ytableausetup{mathmode,boxsize=2em}
\begin{ytableau}
1_1& 1_2 & \none[\dots] & 1_{n-d} & 2_1 &2_2 &\none[\dots]&2_{d-2}\\
2_{d-1}&2_d
\end{ytableau}
$$
There are only 2 more row-semistandard $\mu$-tableaux of type $\lambda$. One with first row with entries $1_1,\dots,1_{n-d-1},2_1,\dots,2_{d-1}$ and second row $1_{n-d},2_d$. The corresponding minimal length double coset representative is $w_1\in S_n$ where $w_1(i)=i-1$ if $n-d+1\le i\le n-1$, $w_1(n-d)=n-1$ and $w_1(i)=i$ otherwise. We see that $w_1=s_{n-2}s_{n-3}\cdots s_{n-d}$.
The second one has first row with entries $1_1,\dots,1_{n-d-2},2_1,\dots,2_d$ and second row $1_{n-d-1},1_{n-d}$. The corresponding minimal length double coset representative is $w_2\in S_n$ where $w_2(i)=i-2$ if $n-d+1\le i\le n$, $w_2(n-d-1)=n-1$, $w_2(n-d)=n$ and $w_2(i)=i$ otherwise. We see that $w_2=s_{n-2}s_{n-3}\cdots s_{n-d-1}s_{n-1}s_{n-2}\cdots s_{n-d}$.

To adjust the answers for
$$
S_{\{1,\dots,d\}}\times S_{\{d+1,\dots,n\}}\backslash S_n/ S_{\{1,2\}}\times S_{\{3,\dots,n\}}
$$
we apply the involution $s_i\leftrightarrow s_{n-i}$ 
and to adjust for a left action as opposed to a right action, we take inverses of the results. After these calculations, we obtain the desired permutations. 
\end{proof}

\begin{lem}
Let $n\ge 2d$, $d\ge 2$. Let $\Sigma=\{\alpha_i:1\le i\le n-1\}$ be the set of simple roots of $\rA_{n-1}$. Let $W=S_n$, $I=\Sigma-\{\alpha_d\}$, and $J=\Sigma-\{\alpha_2\}$. For each of the minimal length double coset representatives of $W_I\backslash W/ W_J$, we compute $\ell(w)$ and the set $J_w$, as defined in Theorem \ref{thm:brosnan}:

\begin{itemize}

\item $w_0=\id$, $J_{w_0}=I\cap J=I-\{\alpha_2\}$, $\ell(w_0)=0$.

\item $w_1=s_ds_{d-1}\cdots s_2$, $J_{w_1}=I-\{\alpha_1,\alpha_{d+1}\}$,
$\ell(w_1)=d-1$.

\item $w_2=s_ds_{d-1}\cdots s_1s_{d+1}s_d\cdots s_2$,
 $J_{w_2}=I-\{\alpha_{d+2}\}$ and $\ell(w_2)=2d$.
\end{itemize}









\end{lem}

\begin{proof}
Since $J_w=\{\alpha\in I: w^{-1}\alpha\in R_J\}$ where $R_J$ is the $\Z$ span of roots in $J$, it suffices to compute $\{w^{-1}\alpha:\alpha\in I\}$ for each case.

It is clear that $J_{\id}=I\cap J$ and $\ell(\id)=0$. For $w=w_1=s_ds_{d-1}\cdots s_2$, we see that $w(1)=1$, $w(2)=d+1$, $w(i)=i-1$ for $3\le i\le d+1$ and $w(i)=i$ for $i\ge d+2$.  Then
$w^{-1}(\alpha_1)=\alpha_1+\alpha_2\not\in R_J$, $w^{-1}(\alpha_{d+1})=\sum_{i=2}^{d+1}\alpha_i\not\in R_J$, but $w^{-1}(\alpha_i)=\alpha_{i+1}\in R_J$ for $2\le i\le d-1$ and $w^{-1}(\alpha_i)=\alpha_i\in R_J$ if $i\ge d+2$. So $J_w=I-\{\alpha_1,\alpha_{d+1}\}$.  

For $w=w_2=s_ds_{d-1}\cdots s_1s_{d+1}s_d\cdots s_2$, we see that $w(1)=d+1,w(2)=d+2,w(i)=i-2$ for $3\le i\le d+2$, and $w(i)=i$, for $i\ge d+3$. So $w^{-1}(\alpha_i)=\alpha_{i+2}\in R_J$ for $1\le i\le d-1$, $w^{-1}(\alpha_{d+1})=\alpha_1\in R_J$, $w^{-1}(\alpha_{i})=\alpha_i\in R_J$ for $i\ge d+3$ but $w^{-1}(\alpha_{d+2})=\sum_{i=2}^{d+2}\alpha_i\not\in R_J$. So $J_w=I-\{\alpha_{d+2}\}$.
\end{proof}

\begin{prop}
Let $A$ be a central simple $F$-algebra of index $d$ and degree $n=kd$ where $k,d\ge 2$.

Let $D=\End_A(\fA)$ for a right ideal $\fA$ of reduced rank $d$ and $B=\End_A(\fA^{k-1})\cong M_{k-1}(D)$ for a complementary right ideal $\fB=\fA^{k-1}$ in $A$. Here, $D$ is a division $F$-algebra of index $d$ and $B$ is a central simple $F$-algebra of index $d$ and degree $(k-1)d$.

The Chow motive of $\SB(2,A)$ decomposes as 
$$
M(\SB(2,A))=M(\SB(2,D))\oplus M(\SB(D)\times \SB(B))(d-1)\oplus M(\SB(2,B))(2d)
$$
\end{prop}

\begin{proof}

This is an application of Brosnan's Theorem 7.1 to the projective
homogeneous variety $\SB(2,A)$ for the group $G=\PGL(A)$. Let $\fA\subset
A$ be a right ideal of reduced rank $d$, and $\fB\subset B$ be a
right ideal of reduced rank $(k-1)d$ such that $A=\fA\oplus \fB$.  Then $D=\End_A(\fA)$ is a division algebra and $B=\End_A(\fB)\cong M_{k-1}(D)$. 

For the cocharacter $\lambda:\Gm\to \PGL(D)\times \PGL(B)$, $t\mapsto \overline{\diag(t\id_{D},\id{B})}$, we have found that the associated parabolic subgroup $P(\lambda)$ is $P_I$ where $I=\Sigma-\{\alpha_d\}$. $\SB(2,A)$ is a projective homogeneous variety for $\PGL(A)$ with respect to the parabolic subgroup $P_J$ where $J=\Sigma-\{\alpha_2\}$, and the components of the decomposition are in bijection with the minimal length double coset representatives of $W_I\backslash W/W_J$, which we have already found.

We have also found the shifts which are the lengths of these representatives. It remains only to find the components $Z_w$.  These are projective homogeneous varieties for the Levi subgroup $L_I$ with respect to the parabolic subgroup $P_{J_w}$. The Levi subgroup is $L_I=\PGL(\End_A(\fA))\times \PGL(\End_A(\fB))=\PGL(D)\times \PGL(B)$. The root system of $L_I$ is then $\rA_{d-1}\times \rA_{(k-1)d-1}$. We will write the simple roots of $\rA_{d-1}$ as $\Sigma_1=\{\alpha_1,\dots,\alpha_{d-1}\}$ and those of $\rA_{(k-1)d-1}$ as 
$\Sigma_2=\{\beta_1,\dots,\beta_{(k-1)d-1}\}$ where $\beta_i$ is identified with 
$\alpha_{d+i}\in \Sigma$.


For $w_0=\id$, $P_{J_{w_0}}=(\Sigma_1-\{\alpha_2\}) \times \Sigma_2$, and so $L_I/P_{J_{w_0}}\cong \SB(2,D)$. For $w_1=s_ds_{d-1}\cdots s_2$, $P_{J_{w_1}}=(\Sigma_1-\{\alpha_1\})\times (\Sigma_2-\{\beta_1\})$, and so $L_I/P_{J_{w_1}}\cong \SB(D)\times \SB(B)$. For  $w_2=s_ds_{d-1}\cdots s_1s_{d+1}s_d\cdots s_2$, $P_{J_{w_2}}=\Sigma_1 \times (\Sigma_2-\{\beta_2\})$, and so $L_I/P_{J_{w_2}}\cong \SB(2,B)$. 
\end{proof}

\begin{prop} 
For a central simple $F$-algebra $A$ of index $d$ and degree $kd$, $d,k\ge 2$, the Chow motive of $\SB(2,A)$ decomposes as 
$$
M(\SB(2,A))=\oplus_{i=0}^{k-1} M(\SB(2,D))(2di)\oplus \oplus_{i=d-1}^{(2k-2)d-2}M(\SB(D))(i)^{a_i}
$$
where
$$
a_{i}=\begin{cases}
\lfloor\frac{i+d+1}{2d}\rfloor, &d-1\le i\le kd-2\\
\lfloor\frac{i+d+1}{2d}\rfloor -\lceil\frac{i-(kd-2)}{d}\rceil, &kd-2\le i\le (2k-2)d-2
\end{cases}
$$
\end{prop}

\begin{proof}
From the previous proposition, note that  $A\cong M_k(D)$ and $B\cong M_{k-1}(D)$.  Then recall that, since $B\cong M_{k-1}(D)$ is Brauer equivalent to $D$, $\SB(D)\times \SB(B)$ is a $\Gr(1,\EE)=\Pj(\EE)$ bundle over $\SB(D)$ for the tautological bundle $\EE$ of $\SB(B)$ by \cite[Prop. 6.3,Rem. 6.5]{Ka98}.

By the Projective Bundle Theorem~\cite[p. 457]{Man68} and the fact that $\EE$ is a bundle over $\SB(D)$ of rank $\deg(B)=(k-1)d$, we have
$$
M(\SB(D)\times \SB(B))=\oplus_{i=0}^{(k-1)d-1}M(\SB(D))(i)
$$
Combining this with the previous proposition and induction on $k$, we obtain our result.
\end{proof}

\begin{rem} 
This is a complete motivic decomposition of the Chow motive for $\SB(2,A)$ for $A$ of degree $d=2^r$. Note that Karpenko showed in \cite[Theorem 2.2.1]{Ka95} that the integral motive of $\SB(D)$ is indecomposable for any division algebra $D$. He later proved in \cite{Ka13} that the motive of $\SB(2,D)$ is indecomposable modulo 2 if $D$ is a division algebra of 2-primary index. Interestingly, Semenov and Zhykhovich in \cite[Theorem 6.1]{SZ15}, based on earlier work by Karpenko and Zhykhovich in \cite{Ka95,Ka00,Ka13,Zhy}, showed  that the integral motives of generalised Severi-Brauer varieties that are indecomposable are those of $\SB(D)$ for a division algebra  or $\SB(2,D)$ for a division algebra of 2-primary index.
\end{rem}

\begin{cor}\label{cor:redtoD}
Let $A$ be a central simple $F$-algebra of index $d$ and let $D$ be a Brauer-equivalent division algebra. Then
$$
\CH^s(\SB(2,A))=\CH^s(\SB(2,D)), \text{ for } 2 \leq s \leq \dim(\SB(2, D)). 
$$ 
\end{cor}

\begin{proof}
Note that for a decomposition of Chow motives,
$M(X)=\oplus_{i=1}^rM(X_i)(a_i)$, implies that we
have $\CH^s(X)=\oplus_i^r \CH^{s+a_i}(X_i)$. Recalling that $\CH^k(Y)=0$ if
$k>\dim(Y)$, we see from the proposition and the facts that
$\dim(\SB(2,D))=2(d-2)$ and $\dim(\SB(D))=d-1$ we see that all but the
first term vanishes.
\end{proof}


\subsection{Extending from the primary case}

Recall that a central simple $F$-algebra $A$ satisfies $A\cong \otimes_{p|
\deg(A)}A_p$ where $A_p$ is a central simple $F$-algebra of degree
$p^{v_p(\deg(A))}$. We exploit this decomposition to extend the result of
Proposition~\ref{prop:plucker} to central simple algebras of index dividing
12.

We will also use the notation $B_p$ to denote the $p$-primary part of an
Abelian group $B$.

Following the argument of Karpenko in \cite{Ka98} for the case of $\SB(A)$,
we see that 
$$
\Tors(\CH^2(\SB(2,A)))=\oplus_{p| \deg(A)}\Tors(\CH^2(\SB(2,A_p))
$$

\begin{prop} 
For every prime $p$ dividing $\deg(A)$, the $p$-primary part of $\CH^2(\SB(2,A))$ is the torsion subgroup  of $\CH^2(\SB(2,A_p))$.
\end{prop}

\begin{proof} 
We quickly recall the argument of \cite[Section 1]{Ka98} for
$\CH^2(\SB(A))$ and see that it also applies to $\CH^2(\SB(2,A))$. For a
finite field extension $E/F$, the argument of \cite[Lemma 1.1]{Ka98} shows
that for the restriction map and norm map on codimension 2 Chow groups of
any twisted projective homogeneous variety $X$, $\res_{E/F}\circ N_{E/F}$
is multiplication by $[E:F]$.  The argument only depends on the fact that
the Grothendieck groups $K(X)$ and $K(X_E)$ are torsion-free and have the
same rank~\cite[Theorem 4.2]{Pa94}. As in \cite[Cor 1.2]{Ka98}, it follows
that if $[E:F]$ is not divisible by a given prime $p$, then
$\CH^2(X)_p\cong \CH^2(X_E)_p$. Now setting $X=\SB(2,A)$, we may follow the
argument of \cite[Prop 1.3]{Ka98} to complete the proof. Fixing a prime $p$
and determining a finite field extension $E/F$ of degree prime to $p$ such
that $A_E$ is Brauer equivalent to $(A_p)_E$, we have 
$$
\CH^2(\SB(2,A))_p\cong \CH^2(\SB(2,A_E))_p\cong \CH^2((\SB(2,(A_p)_E)_p)\cong \CH^2(\SB(2,A_p))_p
$$

Here the first and third steps follow from the fact that $[E:F]$ is relatively prime to $p$, and the remark above, and the middle step follows from \ref{cor:redtoD}. Since $\Tors \CH^2(\SB(2,A_p))$ is annihilated by $\ind{A_p}$, we obtain
$$
\CH^2(\SB(2,A_p))_p=\Tors \CH^2(\SB(2,A_p)).
$$
\end{proof}

\begin{thm} \label{cycle thm}
$\CH^2(\SB(2,A))$ is torsion free for all $F$ central simple algebras $A$ of index $d|12$. 
\end{thm}

\begin{proof}
By the previous results, we only need to show that $\CH^2(\SB(2,D))$ is torsion free for a division algebra of index at most 4. The index 2 case follows from the fact that $\dim(\SB(2,D))=0$ if $\ind(D)=2$. The index 3 case follows from the fact that $\SB(2,D)\cong \SB(1,D^{\rm opp})=\SB(D^{\rm opp})$ where $D$ is of index 3. Since $D^{\rm opp}$ is a division algebra of prime index, $\CH^2(\SB(D^{\rm opp}))$ is torsion free.  The index 4 case was covered by Proposition~\ref{prop:plucker}.
\end{proof}

\begin{rem}
Note that by comparison, that $\CH^k(\SB(A))=\CH^k(\SB(D))$~\cite[Cor. 1.3.2]{Ka96} where $A$ is an $F$ central simple algebra and $D$ is a division algebra Brauer equivalent to $A$. Note also that $\ind(A)=8$ is the smallest index of a central simple algebra $A$ in which torsion appears in $\CH^2(\SB(A))$.  Our result shows that if there is torsion in $\CH^2(\SB(2,A))$ for an central simple algebra of index 8, it will not be predicted by the torsion in the usual Severi-Brauer variety.
\end{rem}



\bibliographystyle{plain}

\end{document}